\documentclass[12pt]{amsart}
\usepackage[top=1in,left=1in,bottom=1in,right=1in]{geometry}
\usepackage{amsmath}
\usepackage{amssymb}
\usepackage{amsthm}
\usepackage{tikz}
\usetikzlibrary{arrows}
\usepackage{wrapfig}
\usepackage{graphicx}
\usepackage{epstopdf}
\usepackage{tikz-cd}
\usepackage{url}
\usepackage{scrextend}
\usepackage{colordvi}
\usepackage{color}
\usepackage{verbatim}
\usepackage{amsmath}
\usepackage{amsfonts}
\usepackage{colordvi}
\usepackage{color}
\usepackage{amsthm}
\usepackage{url}
\usepackage{graphicx}
\usepackage{epstopdf}
\usepackage{amssymb}

\newtheorem{theorem}{Theorem}[section]
\newtheorem{lemma}[theorem]{Lemma}
\newtheorem{conjecture}[theorem]{Conjecture}

\newtheorem{question}[theorem]{Question}
\newtheorem{proposition}[theorem]{Proposition}

\newtheorem{corollary}[theorem]{Corollary}
\theoremstyle{definition}

\newtheorem{definition}[theorem]{Definition}
\theoremstyle{remark}
\newtheorem{remark}[theorem]{Remark}
\newtheorem{example}[theorem]{Example}

\newcommand{\C}{\mathcal C}
\newcommand{\D}{\mathcal D}
\newcommand{\E}{\mathcal E}

\newcommand{\U}{\mathcal U}

\newcommand{\R}{\mathbb R}

\DeclareMathOperator{\conv}{conv}

\DeclareMathOperator{\link}{Lk}

\DeclareMathOperator{\Tk}{Tk}

\DeclareMathOperator{\code}{code}

\newcommand{\nerve}{\mathcal N}

\newcommand{\od}{:=}

\newcommand{\blue}[1]{{\color{blue}#1}}

\newcommand{\Code}{\mathbf{Code}}

\newcommand{\ParCode}{\mathbf{P}_\Code}
\renewcommand{\blue}[1]{#1}

\begin{document}
\title{Sunflowers of Convex Open Sets}

\author{R. Amzi Jeffs}
\address{Department of Mathematics.  University of Washington, Seattle, Wa 98195}
\email{rajeffs@uw.edu}

\begin{abstract}
A \emph{sunflower} is a collection of sets $\{U_1,\ldots, U_n\}$ such that the pairwise intersection $U_i\cap U_j$ is the same for all choices of distinct $i$ and $j$. We study sunflowers of convex open sets in $\R^d$, and provide a Helly-type theorem describing a certain ``rigidity" that they possess. In particular we show that if $\{U_1,\ldots, U_{d+1}\}$ is a sunflower in $\R^d$, then any hyperplane that intersects all $U_i$ must also intersect $\bigcap_{i=1}^{d+1} U_i$. 
We use our results to describe a combinatorial code $\C_n$ for all $n\ge 2$ which is on the one hand minimally non-convex, and on the other hand has no local obstructions. Along the way we further develop the theory of morphisms of codes, and establish results on the covering relation in the poset $\ParCode$.
\end{abstract}

\thanks{Jeffs' research is partially supported by  graduate fellowship from NSF grant DMS-1664865}
\date{\today}
\maketitle

\section{Introduction and Motivation}

Helly's theorem states that if every $d+1$ sets in a collection of convex sets in $\R^d$ share a point, then there is a point common to all the sets. Helly's theorem, along with its relatives such as Radon's and Carath\'eodory's theorems, has \blue{numerous} consequences and generalizations, including fractional, colorful, and topological versions (see \cite{hellytoday, hellyrelatives, eckhoff, matousek} for an overview). Our main result is a new theorem in this lineage, which focuses exclusively on convex sets which satisfy the additional topological constraint that they are open.

We examine the structure of sunflowers of convex open sets. A \emph{sunflower} is a collection of sets $\{U_1,\ldots, U_n\}$ such that $U_i\cap U_j$ is nonempty and constant for distinct choices of $i$ and $j$. That is, $U_i\cap U_j = \bigcap_{k=1}^n U_k\neq\emptyset$ for all choices of distinct $i$ and $j$. The sets $U_i$ are called the \emph{petals} of the sunflower, and $\bigcap_{k=1}^n U_k$ is called the \emph{center} of the sunflower. Our main result is the following: 

\begin{theorem}\label{thm:sunflower}
Let $\{U_1,\ldots, U_{d+1}\}$ be a sunflower of convex open sets with center $U$  in $\R^d$. Then any hyperplane which intersects all $U_i$ must also intersect $U$. 
\end{theorem}

\blue{This generalizes \cite[Lemma 3.2]{obstructions}, which is equivalent to Theorem \ref{thm:sunflower} in the $d=2$ case.}

\begin{example}
The figure below shows a sunflower of convex open sets with three petals in $\R^2$. As we can see, any hyperplane (i.e. line) which intersects all petals must also pass through the center of the sunflower. Theorem \ref{thm:sunflower} states that this phenomenon is general. \[
\includegraphics[scale=0.6]{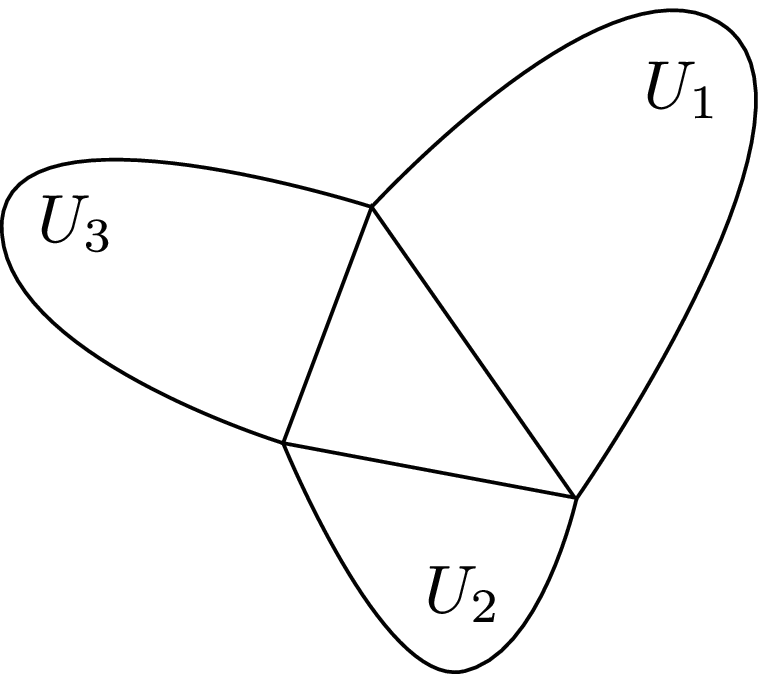}
\]
\end{example}

Theorem 1.1 contrasts typical Helly-type theorems in its topological assumption that all $U_i$ are open sets. This assumption is critical, since without it one can construct a counterexample by taking closed line segments meeting at a point in $\R^2$. \

The second main contribution of this paper is an application of Theorem \ref{thm:sunflower} to the theory of convex codes (sometimes also referred to in the literature as convex neural codes). We give a more detailed overview of convex codes in Section \ref{sec:morphisms}, but for now sketch an outline of our results. A \emph{combinatorial code} or simply \emph{code} is a subset of the Boolean lattice $2^{[n]}$. Informally, a \emph{convex code} is a combinatorial code which records the regions cut out by a collection of convex open sets in some Euclidean space. In \cite{neuralring13} the authors asked which combinatorial codes are convex. This question has been an active area of research in recent years, see for instance \cite{undecidability, openclosed, local15, obstructions}. 

Importantly, almost all known obstructions to convexity are based on ``local" information in a code, in the sense that they depend on links in a simplicial complex associated to a code. The first example of a non-local obstruction to convexity was given by \cite{obstructions}. In Section \ref{sec:Cn} we generalize this example significantly: for every $n\ge 2$, we describe a code $\C_n$ which has no local obstructions to convexity, but nevertheless is not convex. 

An important aspect of the codes $\C_n$ defined in Section \ref{sec:Cn} is that they are ``minimally non-convex," a concept we introduced in \cite{morphisms}. We defer a rigorous explanation of this concept to Section \ref{sec:morphisms}, and for now simply state that a minimally non-convex code is a non-convex code which has the property that any small change to it yields a convex code. Minimally non-convex codes may be thought of as the ``forbidden minors" of convexity. Where before only a single non-local obstruction to convexity was known, we provide an infinite family of such obstructions, and show moreover that these obstructions do not reduce to one another. 

The structure of the paper is as follows. In Section \ref{sec:proof} we prove Theorem \ref{thm:sunflower} and establish several corollaries which are useful in later sections, and may be of interest in their own right. Section \ref{sec:morphisms} is devoted to background on morphisms of codes, along with some new results which make investigating minimally non-convex codes somewhat easier. In particular, Section \ref{sec:morphisms} provides a deeper understanding of a certain covering relation in a poset consisting of combinatorial codes. Section \ref{sec:Cn} defines the family of codes $\C_n$ described above, and uses the results of previous sections to show that each $\C_n$ is minimally non-convex (see Theorem \ref{thm:mnc}). We conclude in Section \ref{sec:conclusion} with some directions for future work. 

\section{Discrete Geometry Results: Theorem \ref{thm:sunflower} and its Corollaries}\label{sec:proof}

Our main goal in this section is to prove Theorem \ref{thm:sunflower}. Below we provide a direct inductive proof, which makes use only of elementary tools. We credit Zvi Rosen with an independent proof of this result, \blue{which} relies on an application of Radon's theorem. 

\begin{proof}[Proof of Theorem \ref{thm:sunflower}]
First observe that we may restrict our attention to the case in which all $U_i$ are bounded by intersecting all $U_i$ with an open ball of sufficiently large radius. We then proceed by induction on $d$. When $d=1$ the claim is immediate, since a hyperplane in $\R^1$ is a single point. 

For $d\ge 2$, we work by contradiction. Suppose that we have a hyperplane $H$ which intersects all $U_i$, but avoids $U$. Since the $U_i$ are open, we may perturb the position and angle of $H$ slightly while maintaining this property. Such perturbations together with the fact that $U$ is bounded allow us to assume without loss of generality that there is a unique point $p$ in $\overline U$ which is closest to $H$. Then $p$ is a boundary point of $U$. We let $H_0$ be the hyperplane parallel to $H$ and passing through $p$ (possibly $H_0 = H$, if $p\in H$ to begin with). 

Note that $H_0\cap U = \emptyset$ since otherwise there would be a point in $\overline U$ which was closer to $H$ than $p$. Moreover, uniqueness of $p$ implies that $H_0\cap \overline U = \{p\}$. Also observe that $H_0$ intersects all $U_i$: choosing $q\in U$ and $p_i\in H\cap U_i$, we see that the line segment between these two points is contained in $U_i$ and intersects $H_0$. 

Next define $\widehat{U_i} = U_i\cap H_0$ for all $i$, and choose points $\hat q_i$ in $\widehat{U_i}$. Then choose an affine subspace $H' \subseteq H_0$ of dimension $d-2$ passing through the points $p$ and $\hat q_1, \ldots, \hat{q}_{d-2}$. Observe that $H' $ has codimension one in $H_0$, and so splits it into two regions. By pigeonhole principle, two of the three points in $\{\hat q_{d-1}, \hat q_d, \hat q_{d+1}\}$ will lie (perhaps not strictly) on one side of $H'$. Say without loss of generality that these two points are $\hat q_{d-1}$ and $\hat q_d$. Since all $\widehat{U_i}$ are relatively open in $H_0$, we may choose $\hat{p_i}\in \widehat{U}_i$ near $\hat{q}_i$ for $1\le i \le d$ so that $\hat p_1, \ldots, \hat{p}_{d}$ all lie \emph{strictly} on one side of $H'$. 

Now, let $\widehat{H}\subseteq H_0$ be a parallel translate of $H'$ so that $\hat p_1, \ldots, \hat{p}_{d}$ lie strictly on one side of $\widehat{H}$, while $p$ lies strictly on the other. We will show that $\widehat{H}$ intersects $\widehat U_i$ for $1\le i \le d$. To see this, consider the line segment from $\hat p_i$ to $p$ for $1\le i\le d$. We claim that every point on this line segment other than $p$ lies in $\widehat U_i$. Choose $q\in U$ and consider the line segment from $q$ to $p$. Since $p$ is a boundary point of $U$ it follows that it is the only point on this line segment not lying in $U$. Let $L$ denote the half-open line segment from $q$ to $p$ which does not include $p$. Then choose a neighborhood of $\hat p_i$ contained in $U_i$, and consider the convex hull of this neighborhood with $L$, as pictured below. \[
\includegraphics[scale=0.6]{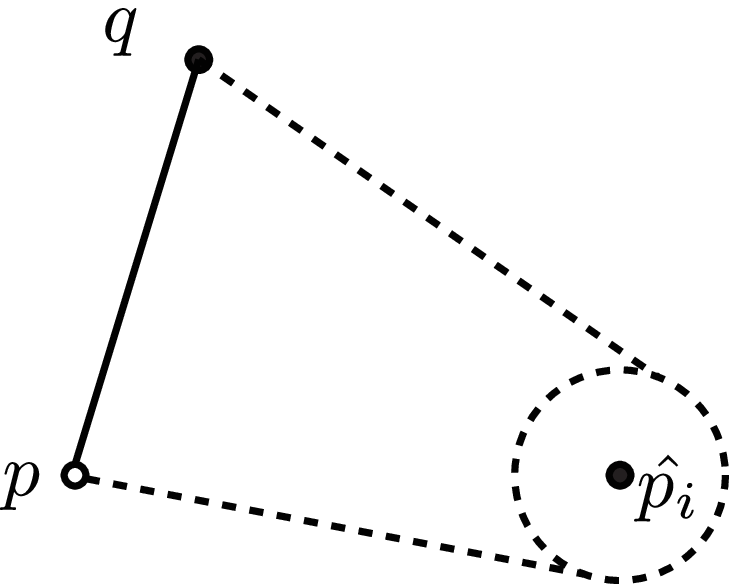}
\]
This convex hull is contained in $U_i$ since $L$ and the neighborhood of $\hat p_i$ are both contained in $U_i$. Moreover, all points strictly between $p$ and $\hat p_i$ are contained in this convex hull, and hence in $U_i$. These points are also contained in $H_0$, and hence in $\widehat{U_i}$. 
Since the line segment from $\hat p_i$ to $p$ crosses $\widehat{H}$, we conclude that $\widehat{H}$ contains a point in $\widehat U_i$.

 Since $p$ is the unique boundary point of $U$ lying in $H_0$, $\widehat{H}$ and $\overline U$ are disjoint. In particular, they are a positive distance from one another. Orient $H_0$ so that $U$ lies on its positive side, and for $\varepsilon >0$  let $H_\varepsilon$ denote the shift of $H_0$ by $\varepsilon$ in the positive direction. Moreover let $\widehat H_\varepsilon$ denote the shift of $\widehat H$ by the same amount, so that in particular $\widehat H_\varepsilon \subseteq H_\varepsilon$. Since $\widehat{H}$ is a positive distance from $\overline U$, we may choose $\varepsilon$ small enough so that $\widehat H_\varepsilon$ does not intersect $U$. Now consider the collection of relatively open convex sets $V_i = U_i\cap H_\varepsilon$ in $H_\varepsilon$. We may choose $\varepsilon$ small enough so that $U\cap H_\varepsilon \neq \emptyset$, which implies that the collection $\{V_1,\ldots, V_d\}$ is a sunflower in $H_\varepsilon\cong \R^{d-1}$. 

We claim that $\widehat H_\varepsilon$ contains points in each $V_i$ but not in their common intersection, contradicting our inductive hypothesis. Indeed, $\widehat H_\varepsilon\cap U = \emptyset$ by choice of $\varepsilon$, so $\widehat H_\varepsilon$ does not contain any points in the common intersection of the $V_i$. To see that $\widehat H_\varepsilon\cap V_i\neq \emptyset$ for small enough $\varepsilon$, we simply use the openness of $U_i$. In particular, $\widehat H$ contains a point in $U_i$ for $1\le i \le d$, so any appropriately small translate of $\hat H$ will have the same property. Thus we may choose a small $\varepsilon >0$ such that $\widehat H_\varepsilon$ contains a point in $U_i$, and hence in $V_i$ for $1\le i \le d$. This contradicts our inductive hypothesis and the result follows. 
\end{proof}

The following corollary may be viewed as a \blue{partial converse to B\'ar\'any's colorful Carath\'eodory theorem \cite{caratheodorygeneralization}.  The colorful Carath\'eodory theorem states that if $P_1,\ldots, P_n$ are sets in $\R^d$ with $n>d$ and $p\in \bigcap_{i=1}^n \conv(P_i)$, then there exist points $p_i\in P_i$ such that $p\in \conv\{p_1,\ldots, p_n\}$. Corollary \ref{cor:conv} says that if $\{\conv(P_1),\ldots, \conv(P_n)\}$ forms an open sunflower, then any choice of $p_i\in P_i$ has the property that $\conv\{p_1,\ldots, p_n\}$ contains a point the common intersection of all $\conv(P_i)$. Roughly, the colorful Carath\'eodory theorem says that any point $p$ lying in every $\conv(P_i)$ can be ``witnessed" as a convex combination of a single point from each $P_i$, while Corollary \ref{cor:conv} says that when we have a sunflower of convex open sets every choice of $p_i\in P_i$ witnesses a point $p$ lying in all $\conv(P_i)$.}

\begin{corollary}\label{cor:conv}
Let $\{U_1,\ldots, U_{n}\}$ be a sunflower of convex open sets with center $U$  in $\R^d$ where $n\ge d+1$, and let $p_i\in U_i$ for $1\le i \le n$. Then $\conv\{p_1,\ldots, p_{n}\}$ intersects $U$. 
\end{corollary}
\begin{proof}
It suffices to prove that $C = \conv\{p_1,\ldots, p_{d+1}\}$ intersects $U$. Suppose for contradiction that $C$ and $U$ are disjoint. Then there is a hyperplane $H$ weakly separating $C$ and $U$. By considering line segments from $p_i$ to points in $U$, we see that $H$ intersects $U_i$ for $1\le i \le d+1$. But by assumption $H$ does not intersect $U$, contradicting Theorem \ref{thm:sunflower}.
\end{proof}

\begin{corollary}\label{cor:plane}
Let $\{U_1,\ldots, U_{n+1}\}$ be a sunflower of convex open sets with center $U$ in $\R^d$, and let $k\le \min\{d-1, n-1\}$. Then any $k$-dimensional affine space $H$ intersecting all $U_i$ must intersect $U$.
\end{corollary}
\begin{proof} 
Since $k\le n-1$, we may delete sets in our sunflower until we have $k= n-1$. It suffices to prove the result in this case. Since $k \le d-1$, the space $H$ has positive codimension in $\R^d$ and we may consider a $(k+1)$-dimensional affine subspace $A$ containing $H$ and a point $p\in U$. Setting $V_i = U_i\cap A$, we obtain a sunflower $\{V_1,\ldots, V_{n+1}\}$ in $A\cong \R^{n}$. The hyperplane $H\subseteq A$ intersects all $V_i$, and so by Theorem \ref{thm:sunflower} $H$ intersects the center of $\{V_1,\ldots, V_{n+1}\}$. This center is contained in $U$, so $H$ intersects $U$ and the result follows.
\end{proof}

Theorem \ref{thm:sunflower} and the corollaries above echo Helly-type theorems, with the critical difference that our results apply only to open sets. It is natural to wonder whether there exist other Helly-type results that are specific to open sets. We leave this question open, and for now turn to applications of the above results to the theory of convex codes.

\section{Convex Codes, Morphisms, and Covering Relations in $\ParCode$}\label{sec:morphisms}

Our aim throughout the rest of the paper is to apply Theorem \ref{thm:sunflower} to the study of convex codes. We begin by recalling some standard definitions regarding codes. For a comprehensive review, see \cite{morphisms, neuralring13}. 

\begin{definition}
A \emph{code} or \emph{combinatorial code} is a subset of the Boolean lattice $2^{[n]}$. 
\end{definition}

\begin{definition}
Let $\U = \{U_1,\ldots, U_n\}$ be a collection of convex open sets in $\R^d$, and let $X\subseteq \R^d$ be a convex open set with $U_i\subseteq X$ for all $i$. We define the \emph{code} of $\{U_i\}_{i=1}^n$ in $X$ to be the code \[
\code(\{U_i\}_{i=1}^n, X) = \bigg\{\sigma\subseteq [n] \ \bigg|\  \bigcap_{i\in\sigma} U_i\setminus \bigcup_{j\notin \sigma}U_j \neq \emptyset\bigg\},
\]
where by convention the empty intersection is $X$.  The set $X$ is called the \emph{ambient space}, and the collection $\{U_i\}_{i=1}^n$ is called a \emph{convex realization} of $\code(\{U_i\}_{i=1}^n, X)$ in the ambient space $X$. If a code $\C$ has a convex realization, then we say that $\C$ is \emph{convex}. Given $\sigma\subseteq 2^{[n]}$, the region $\bigcup_{i\in\sigma} U_i\setminus \bigcup_{j\notin \sigma}U_j$ is called the \emph{atom} of $\sigma$. 
\end{definition}

One can think of a convex code as  recording the labels on the non-empty atoms of a collection of convex sets in $\R^d$. Classifying convex codes has been an active area of study in recent years, see for example \cite{undecidability, openclosed, local15, neuralring13, obstructions}. Despite this work, convex codes are very far from fully classified or understood. In \cite{morphisms} we introduced a notion of morphism of combinatorial codes, with the goal of isolating certain minimal obstructions to convexity. In this section we further develop this theory, with the aim of better describing the structure of the family of codes presented in Section \ref{sec:Cn}. We begin by recalling some of the definitions and results of \cite{morphisms}.

\begin{definition}\label{def:trunk}
Let $\C\subseteq 2^{[n]}$ be a code. For any $\sigma\subseteq[n]$, the \emph{trunk} of $\sigma$ in $\C$ is the set \[
\Tk_\C(\sigma) \od \{c\in \C \mid \sigma\subseteq c\}.
\] A subset of $\C$ is called a \emph{trunk in $\C$} if it is empty, or equal to $\Tk_\C(\sigma)$ for some $\sigma\subseteq[n]$. A trunk of the form $\Tk_\C(\{i\})$ will be called a \emph{simple trunk}, and denoted $\Tk_\C(i)$. 
\end{definition}

\begin{definition}\label{def:morphism}
Let $\C\subseteq 2^{[n]}$ and $\D\subseteq 2^{[m]}$ be codes. A function $f:\C\to\D$ is a \emph{morphism} if for every trunk $T$ in $\D$ the preimage $f^{-1}(T)$ is a trunk in $\C$. 
\end{definition}

The class of  codes together with morphisms forms a category $\Code$. This category provides useful information about convexity of codes. In particular, convexity is an isomorphism invariant, and in fact the image of a convex code under a morphism is again convex. Moreover, if $\C$ is a convex code, then so are all the trunks of $\C$. These results motivate the formation of a poset of combinatorial codes, in which codes are comparable to one another via surjective morphisms and taking trunks. We make this notion precise below. 

\begin{definition}\label{def:PCode}
An \emph{operation} on a code $\C$ consists of replacing $\C$ by one of its trunks, or by its image under some morphism. 
For isomorphism classes of codes $[\C]$ and $[\D]$, we say that $[\C]\le [\D]$ if there is a series of operations taking $\D$ to $\C$. The relation $[\C]\le [\D]$ forms a partial order on isomorphism classes of codes. The poset consisting of isomorphism classes of codes ordered by this relation is denoted $\ParCode$. 
\end{definition}

One of the main results of \cite{morphisms} is the following: 
\begin{theorem}[\cite{morphisms}]
The collection of isomorphism classes of convex codes forms a down-set in $\ParCode$. Equivalently, if $[\C]\le [\D]$ and $\D$ is convex, then so is $\C$. 
\end{theorem}

\begin{definition}
Let $\C$ be a code. We say that $\C$ is \emph{minimally non-convex} if $\C$ is not a convex code, but every isomorphism class $[\D]< [\C]$ is convex.
\end{definition}

Minimally non-convex codes are the minimal elements of the complement of the down-set in $\ParCode$ consisting of convex codes. They can be thought of as the codes that form the ``boundary" between convexity and non-convexity in $\ParCode$. Thus to characterize all convex codes, it suffices to characterize these minimal obstructions to convexity. 

To prove that the family of codes presented in Section \ref{sec:Cn} are all minimally non-convex, we require an understanding of the covering relation in $\ParCode$. We develop this understanding below after reviewing some additional results from \cite{morphisms}. 

Given a collection of trunks $\{T_1,\ldots, T_m\}$ in a code $\C$, one can use them to define a morphism from $\C$ to $2^{[m]}$ by $c\mapsto \{j\in[m]\mid c\in T_j\}$. We will speak of this morphism as the \emph{morphism determined by} the collection of trunks $\{T_1,\ldots, T_m\}$. Sometimes we will abuse notation by ignoring the labels on the trunks in a collection, since up to permutation the labels do not affect the image of the code $\C$.

\blue{A useful fact} is that every morphism is determined by a set of trunks, as described by the following proposition.

\begin{proposition}[\cite{morphisms}]\label{prop:determined}
Let $\C\subseteq 2^{[n]}$ and $\D\subseteq 2^{[m]}$, and let $f:\C\to \D$ be a morphism. For $j\in[m]$ let $T_j = f^{-1}(\Tk_{\D}(j))$. Then for all $c\in \C$,\[
f(c) = \{j\in[m]\mid c\in T_j\}.
\]
That is, $f$ is the morphism determined by the collection $\{T_1,\ldots, T_m\}$. 
\end{proposition} 

\blue{As a consequence, every morphism is a \emph{monotone} map: if $c_1\subseteq c_2$, then $f(c_1)\subseteq f(c_2)$.} A further useful fact is that the intersection of two trunks is again a trunk. In the rest of this section, we use intersections of trunks to build our understanding of the covering relation in $\ParCode$. We begin with a few definitions.

\begin{definition}
Let $\C$ be a code, and let $\{T_j\}$ be a collection of trunks in $\C$. A trunk $T\subseteq \C$ is \emph{generated} by $\{T_j\}$ if $T$ can be written as an intersection of various $T_j$. 
\end{definition}

\begin{definition}[\cite{morphisms}]\label{def:irreducibletrunk}
Let $\C$ be a code. A trunk $T\subseteq \C$ is called \emph{irreducible} if it is not the intersection of two trunks that properly contain it. Equivalently, $T$ is irreducible if every set of trunks that generates $T$ has $T$ as an element. 
\end{definition}

\begin{definition}[\cite{morphisms}]\label{def:reduced}
Let $\C\subseteq 2^{[n]}$ be a code. An index $i$ is called \emph{trivial} if $\Tk_\C(i) = \emptyset$. An index $i$ is called \emph{redundant} if there exists $\sigma\subseteq [n]\setminus \{i\}$ such that $\Tk_\C(i) = \Tk_\C(\sigma)$ (equivalently, if $\Tk_\C(i)$ is generated by $\{\Tk_\C(j)\mid j\neq i\}$). A code is called \emph{reduced} if it has no trivial or redundant indices. 
\end{definition}

A key result of \cite{morphisms} is that every code is isomorphic to a unique (up to permutation of indices) reduced code. Thus one can often restrict one's attention to reduced codes without loss of generality. A second useful result from \cite{morphisms} is that in a reduced code every simple trunk is irreducible. Below we establish \blue{two key lemmas} which we will use to investigate the structure of $\ParCode$.

\blue{
\begin{lemma}\label{lem:foftrunkintersection}
Let $f:\C\to \D$ be a morphism determined by trunks $\{T_j\}$. For any set of indices $\sigma$, $f\left(\bigcap_{j\in\sigma} T_j\right) = \bigcap_{j\in\sigma} f(T_j)$. 
\end{lemma}
\begin{proof}
The inclusion $f\left(\bigcap_{j\in\sigma} T_j\right) \subseteq \bigcap_{j\in\sigma} f(T_j)$ is immediate for any function $f$. For the reverse inclusion, let $d\in\bigcap_{j\in\sigma} f(T_j)$, and let $c\in \C$ be such that $f(c) = d$. Recall from Proposition \ref{prop:determined} that $f(T_j) = \Tk_\D(j)$, and so $d\in \bigcap_{j\in\sigma}\Tk_\D(j) = \Tk_\D(\sigma)$. Thus $\sigma\subseteq d$. But again applying Proposition \ref{prop:determined}, we see that $d = f(c) = \{j\mid c\in T_j\}$. Since $\sigma\subseteq d$, this implies that $c\in T_j$ for all $j\in \sigma$. Thus $d\in f\left(\bigcap_{j\in\sigma} T_j\right)$ as desired.
\end{proof}}

\begin{lemma}\label{lem:factor}
Let $\C$ be a code, and $f:\C\to \D$ and $g:\C\to \E$ be surjective morphisms determined by trunks $\{T_j\}$ and $\{S_k\}$ respectively. Then there exists a morphism $h:\D\to\E$ such that $g = h\circ f$ if and only if every $S_k$ is generated by $\{T_j\}$.
\end{lemma}
\begin{proof}
$(\Rightarrow)$ Suppose that $h:\D\to\E$ is such that $g = h \circ f$. For each $k$ we have that \[
S_k = g^{-1}(\Tk_\E(k)) = f^{-1}(h^{-1}(\Tk_\E(k))).
\]
But $h^{-1}(\Tk_\E(k))$ is a trunk in $\D$, so it can be written as an intersection of simple trunks in $\D$. That is, there exists a set of indices $\sigma$ such that $h^{-1}(\Tk_\E(k)) = \bigcap_{j\in\sigma} \Tk_\D(j)$. Putting this into our expression above we get \[
S_k = f^{-1}\left(\bigcap_{j\in\sigma} \Tk_\D(j)\right) = \bigcap_{j\in\sigma} f^{-1}(\Tk_\D(j)) = \bigcap_{j\in\sigma} T_j.
\]
Thus $S_k$ is generated by $\{T_j\}$ as desired.

$(\Leftarrow)$ Suppose that each $S_k$ is generated by $\{T_j\}$. For each $k$ fix a set of indices $\sigma_k$ such that $S_k  = \bigcap_{j\in \sigma_k} T_j$. Define $Q_k = f(S_k)\subseteq \D$. We claim that each $Q_k$ is a trunk in $\D$. To see this, note  that \[
Q_k = f(S_k) = f\left(\bigcap_{j\in \sigma_k} T_j\right) =  \bigcap_{j\in \sigma_k} f(T_j),
\]
\blue{where the last equality follows from Lemma \ref{lem:foftrunkintersection}.} But $f(T_j) = \Tk_\D(j)$ by surjectivity of $f$, so the last term above is an intersection of trunks in $\D$. Thus $Q_k$ is a trunk.

 Let $n$ be such that $\{Q_k\} = \{f(S_k)\}$ is indexed by $[n]$, and note that $\E\subseteq 2^{[n]}$. Let $h:\D \to 2^{[n]}$ be the morphism determined by $\{Q_k\}$. We claim that $h(\D) \subseteq \E$ so that we may regard $h$ as a morphism from $\D$ to $\E$, and further that $g = h\circ f$ with this restriction. To see this, note that for $c\in \C$ \[
g(c) = \{k \mid c\in S_k\} = \{k \mid f(c)\in f(S_k)\} = \{k \mid f(c)\in Q_k\} = h(f(c)).
\]
The second equality above follows from the fact that $S_k$ is generated by $\{T_j\}$ and that $f$ is determined by $\{T_j\}$. This proves the result. 
\end{proof}

\begin{corollary}\label{cor:redundant}
Let $\C$ be a code and $\{T_j\}$ a collection of trunks in $\C$. Suppose that $T$ is a trunk generated by $\{T_j\}$, and let $f$ and $g$ be the morphisms determined by the collections $\{T_j\}$ and $\{T_j\}\cup \{T\}$ respectively. Then $f(\C)$ and $g(\C)$ are isomorphic.
\end{corollary}
\begin{proof}
The above lemma yields surjective maps $h_1:f(\C)\to g(\C)$ and $h_2:g(\C) \to f(\C)$ which are mutual inverses. This gives the desired isomorphism. 
\end{proof}

Consider a morphism determined by trunks $\{T_j\}$ in a code $\C$. By the above corollary, the image of $\C$, up to isomorphism, can be obtained by repeatedly deleting various $T_j$ which are generated by $\{T_k\mid k\neq j\}$ and then considering the image of $\C$ under the morphism determined by the resulting collection. That is, we can remove various $T_j$ which are are ``redundant" to the other trunks in the collection. This will be used in our analysis in Section \ref{sec:Cn}.

 A further consequence of Lemma \ref{lem:factor} is a description of the covering relation in $\ParCode$, which we develop in Theorems \ref{thm:converse} and \ref{thm:cover}. We begin by describing some seemingly tangential properties of $\ParCode$, which will turn out to be useful later on.

\begin{proposition}\label{prop:nomoretrunks}
Let $f:\C\to\D$ be a surjective morphism of codes. Then the map $T\mapsto f^{-1}(T)$ is an injective map from the set of trunks in $\D$ to the set of trunks in $\C$. In particular, $\C$ has at least as many trunks as $\D$. 
\end{proposition}
\begin{proof}
Suppose that $T,S\subseteq \D$ are both trunks such that $f^{-1}(T) = f^{-1}(S)$. Surjectivity of $f$ then implies that $T = f(f^{-1}(T)) = f(f^{-1}(S)) = S$, proving the result. 
\end{proof}

\begin{proposition}\label{prop:surjectiveisom}
Let $f:\C\to \D$ be a surjective morphism of codes, and suppose that $\C$ and $\D$ have the same number of trunks. Then $f$ is an isomorphism.
\end{proposition}
\begin{proof}
We first prove that the map $T\mapsto f^{-1}(T)$ is a bijection on trunks. Proposition \ref{prop:nomoretrunks} states that this map is an injection, and since $\C$ and $\D$ have the same number of trunks we conclude it is a bijection. 

Since we know that $f$ induces a bijection on trunks, it suffices to prove that $f$ is a bijective function. By hypothesis $f$ is surjective, so we need only show injectivity. \blue{If $c_1\neq c_2$, then $\Tk_\C(c_1) \neq \Tk_\C(c_2)$, and since $f$ induces a bijection on trunks we conclude that $f(\Tk_\C(c_1)) \neq f(\Tk_\C(c_2))$. But since $f$ is monotone with respect to the partial order on $\C$ and $\D$ and $c_i$ is the unique minimal element of $\Tk_\C(c_i)$, we see that $f(c_i)$ is the unique minimal element of $f(\Tk_\C(c_i))$. In particular, $f(\Tk_\C(c_i)) = \Tk_\D(f(c_i))$. Thus $\Tk_\D(f(c_1)) \neq \Tk_\D(f(c_2))$, so $f(c_1)\neq f(c_2)$, and $f$ is injective as desired.}
\end{proof}

\begin{corollary}\label{cor:fewertrunks}
If $[\D] <[\C]$ in $\ParCode$, then $\D$ has strictly fewer trunks than $\C$. 
\end{corollary}\begin{proof}
It suffices to prove the result in the case that $[\C]$ covers $[\D]$. In this case, either there is a surjective morphism $f:\C\to \D$ that is not an isomorphism, or $\D$ is isomorphic to a trunk in $\C$. In the former case observe that $\D$ has no more trunks than $\C$ by Proposition \ref{prop:nomoretrunks}. Moreover, since the map $\C\to \D$ is not an isomorphism, Proposition \ref{prop:surjectiveisom} implies that $\D$ must have strictly fewer trunks than $\C$. In the latter case, $\D$ is isomorphic to a trunk $T$ in $\C$. It suffices to show that $T$ has fewer trunks than $\C$. Since $T$ is not isomorphic to $\C$ it must be a proper trunk. Every trunk in $T$ is also a trunk in $\C$, but this association is not surjective since for any $c\in \C\setminus T$ the trunk $\Tk_\C(c)$ is not a trunk in $T$. Thus $T$ has fewer trunks than $\C$, proving the result. 
\end{proof}

\begin{corollary}\label{cor:onelesscovered}
If $[\D]<[\C]$ in $\ParCode$ and $\D$ has exactly one less trunk than $\C$, then $[\C]$ covers $[\D]$. 
\end{corollary}
\begin{proof}
Suppose not. Then there exists a code $\E$ such that $[\D]<[\E]<[\C]$. By Corollary \ref{cor:fewertrunks} the number of trunks in $\D$ is less than the number in $\E$, which in turn is less than the number of trunks in $\C$. This is not possible since $\D$ has exactly one less trunk than $\C$.
\end{proof}

\begin{definition}\label{def:coveredcode}
Let $\C\subseteq 2^{[n]}$ be a reduced code and let $T_1,\ldots, T_n$ be the simple trunks in $\C$. For $i\in[n]$, the  \emph{$i$-th covered code} of $\C$, denoted $\C^{(i)}$, is the image of $\C$ under the morphism determined by the  following collection of trunks: \[\{T_j \mid j\neq i\} \cup \{T_j\cap T_i\mid j\neq i \text{ and } T_j\cap T_i \neq T_i\}.\] 
\end{definition}

The following theorem justifies this terminology.

\begin{theorem}\label{thm:converse}
Let $\C\subseteq 2^{[n]}$ be a reduced code and $i\in [n]$. Then $[\C^{(i)}]$ is covered by $[\C]$ in $\ParCode$. 
\end{theorem}
\begin{proof}
By Corollary \ref{cor:onelesscovered} it suffices to show that $\C^{(i)}$ has exactly one less trunk than $\C$. Let $f:\C\to \C^{(i)}$ be the surjective morphism defining $\C^{(i)}$, and for $j\in [n]$ let $T_j = \Tk_\C(j)$. We claim that the map $T\mapsto f^{-1}(T)$ is a bijection from trunks in $\C^{(i)}$ to the trunks in $\C$ that are not equal to $T_i$. 

Observe by Proposition \ref{prop:nomoretrunks} that this map is injective from trunks in $\C^{(i)}$ to trunks in $\C$. Suppose for contradiction that there exists $T\subseteq \C^{(i)}$ so that $f^{-1}(T) = T_i$. Since $\C$ is reduced $T$ must be nonempty. In particular, there exists $\sigma$ so that $T =\Tk_{\C^{(i)}}(\sigma)$. But then  \[
T_i = f^{-1}(T) = \bigcap_{k\in \sigma} f^{-1}(\Tk_{\C^{(i)}}(k)).
\]
By definition of $f$, each term $f^{-1}(\Tk_{\C^{(i)}}(k))$ is equal to either $T_j$ for some $j\neq i$, or equal to $T_j\cap T_i$ for some $j\neq i$ with $T_j\cap T_i \neq T_i$. Thus the above equality implies that $\Tk_\C(i)$ is generated by the set of trunks $\{T_j \mid j\neq i\} \cup \{T_j\cap T_i\mid j\neq i \text{ and } T_j\cap T_i \neq T_i\}$. This set of trunks does not include $T_i$ as an element, but by \cite[Theorem 2.7] {morphisms} $T_i$ is an irreducible trunk. This is a contradiction.

So far we have shown that $T\mapsto f^{-1}(T)$ is an injective map from trunks in $\C^{(i)}$ to trunks in $\C$ that are not equal to $T_i$. To argue surjectivity, let $S\subseteq \C$ be a trunk that is not equal to $T_i$. If $S$ is empty, then it is the preimage of the empty trunk. If $S$ is nonempty, then observe that $S$ can be written as a (possibly empty) intersection of trunks in the set $\{T_j \mid j\neq i\} \cup \{T_j\cap T_i\mid j\neq i \text{ and } T_j\cap T_i \neq T_i\}$. The image of any trunk in this set is again a trunk in $\C^{(i)}$, and so $S$ is the preimage of the trunk in $\C^{(i)}$ arising from the corresponding intersection of images of trunks. This proves that $T\mapsto f^{-1}(T)$ is a bijective map between the trunks in $\C^{(i)}$ and the trunks in $\C$ that are not equal to $T_i$. Thus $\C^{(i)}$ has exactly one less trunk than $\C$, and the result follows. 
\end{proof}

\begin{theorem}\label{thm:cover}
Let $\C$ and $\D$ be codes. If $[\C]$ covers $[\D]$ in $\ParCode$ then \begin{itemize}
\item[(i)] $\D$ is isomorphic to a simple trunk in $\C$, or 
\item[(ii)] $\D$ is isomorphic to $\C^{(i)}$ for some $i$.
\end{itemize}
\end{theorem}
\begin{proof}
Throughout we will let $T_1,\ldots, T_n$ denote the simple trunks of $\C$. Suppose that $[\C]$ covers $[\D]$ but (i) above does not hold. Then $\D$ must be the non-isomorphic image of $\C$ under a morphism. Let $f:\C\to \D$ denote such a morphism, and consider the trunks that determine $f$. Among these trunks, there must be some $T_i$ which does not appear, lest $f$ be an isomorphism. Then every trunk determining $f$ is an intersection of various $T_j$, and whenever $T_i$ appears in such an intersection the intersection is not equal to $T_i$. We then see by Lemma \ref{lem:factor} that the morphism $f$ factors through the map $\C\to \C^{(i)}$. Since $\D$ is covered by $\C$, the factoring map must be an isomorphism and $\D\cong \C^{(i)}$ as desired. 
\end{proof}

\begin{remark}
It is important to note that the converse of the above theorem does not hold in general. Although Theorem \ref{thm:converse} states that each $[\C^{(i)}]$ is covered by $[\C]$ when $\C$ is reduced, the same may not be true of trunks. Consider the code $\C = \{23, 2, 3, 1, \emptyset\}$. Here we have $\Tk_\C(1) = \{1\}$, and $[\C]$ does not cover $[\Tk_\C(1)]$. This follows from the fact that $\Tk(2) = \{2,23\}$, and the isomorphism class of $\Tk(2)$ trunk lies  strictly between $[\C]$ and $[\Tk_\C(1)]$ in $\ParCode$. 
\end{remark}

\begin{remark}
The arguments used in Theorems \ref{thm:converse} and \ref{thm:cover} imply that if we partially order isomorphism classes of codes via surjective maps (as opposed to surjective maps \emph{and} replacement by trunks, as in Definition \ref{def:PCode}), then $[\C]$ covers $[\D]$ if and only if $\D\cong \C^{(i)}$ for some $i$. The resulting poset is then graded, with the rank function being equal to the number of trunks in a code. 
\end{remark}

We require one last tool before proceeding to Section \ref{sec:Cn}. Recall that a code is max-intersection complete if it contains all intersections of maximal codewords. The lemma below describes when the image of a code is max-intersection complete. This proves useful because by \cite{openclosed}, a max-intersection complete code is always convex. For the proof below, recall from \cite{morphisms} that morphisms respect the partial order on codewords given by inclusion, so that in particular maximal codewords are mapped to maximal codewords in the image. Also recall that for any code $\C$, its intersection completion is denoted $\widehat{\C}$ and the set of its maximal codewords is denoted $M(\C)$.  

\begin{lemma}\label{lem:imagemic}
Let $\C\subseteq 2^{[n]}$ and choose $\sigma_1,\ldots, \sigma_m$ with $\sigma_j\subseteq[n]$ for all $j\in[m]$. For $j\in[m]$ define $T_j = \Tk_\C(\sigma_j)$, and let $f:\C\to 2^{[m]}$ be the morphism determined by the collection $\{T_j\}$. Then $f(\C)$ is max-intersection complete if and only if the following holds: for every $\sigma \in \widehat{M(\C)}\setminus \C$ there exists $c\in \C$ such that $\sigma_j\subseteq c$ if and only if $\sigma_j\subseteq \sigma$ for all $j\in[m]$. 
\end{lemma}
\begin{proof}
$(\Rightarrow)$ Suppose that $f(\C)$ is max-intersection complete, and let $\sigma\in \widehat{M(\C)}\setminus \C$. Then $\sigma$ is the intersection of maximal codewords $c_1,\ldots, c_k$ in $\C$. Let $d\in f(\C)$ be the intersection of the various $f(c_i)$, and let $c\in \C$ be such that $f(c) = d$. We claim that $\sigma_j\subseteq c$ if and only if $\sigma_j\subseteq \sigma$. First note that $\sigma_j\subseteq c$ if and only if $j\in d$. This happens if and only if $j\in f(c_i)$ for all $i$, which in turn is equivalent to $c_i\in T_j$ for all $i$. But this is equivalent to $\sigma_j\subseteq \sigma$ since $\sigma$ is the intersection of the $c_i$. The result follows.

$(\Leftarrow)$ Let $d_1,\ldots, d_k$ be maximal codewords in $f(\C)$, and let $\tau$ be their intersection. Let $c_1,\ldots, c_k$ be such that $f(c_i) = d_i$, and let $\sigma$ be the intersection of the $c_i$. By hypothesis there exists $c\in \C$ such that $\sigma_j\subseteq c$ if and only if $\sigma_j\subseteq\sigma$. We claim that $f(c) = \tau$, so that $f(\C)$ is max-intersection complete. First observe that  $j\in \tau$ if and only if $j\in d_i$ for all $i$, which is equivalent to $\sigma_j\subseteq c_i$ for all $i$. This happens if and only if $\sigma_j\subseteq \sigma$, which by choice of $c$ is equivalent to $j\in f(c)$. This proves the result.
\end{proof}

\begin{corollary}\label{cor:Cimic}
Let $\C\subseteq 2^{[n]}$ and suppose that $\widehat{M(\C)}\setminus \C$ contains only the codeword $\{i\}$ for some $i$. Then $\C^{(j)}$ is max-intersection complete if and only if $j=i$ and $\emptyset\in \C$, or $j\neq i$ and $\{i,j\}\in\C$.
\end{corollary}

\section{A Family of Locally Perfect Minimally Non-Convex Codes}\label{sec:Cn}
In this section we describe a family of codes (see Definition \ref{def:Cn}) which are minimally non-convex (see Theorems \ref{thm:nonconvex} and \ref{thm:mnc}), and have obstructions that arise from Theorem \ref{thm:sunflower}. 
Importantly, these codes have no other obstructions to convexity, as we will discuss below (see Theorem \ref{thm:locallyperfect}). This family of codes provides the first infinite collection of minimally non-convex codes that do not have local obstructions or nerve obstructions. 

Before proceeding we recall a couple of tools regarding simplicial complexes. Given a code $\C\subseteq 2^{[n]}$, the smallest simplicial complex containing $\C$ is denoted $\Delta(\C)$. Given a collection of convex open sets $\U = \{U_1,\ldots, U_n\}$ in $\R^d$, the \emph{nerve} of $\U$ is the simplicial complex \[
\nerve(\U) \od \{\sigma\subseteq [n] \mid U_\sigma\neq\emptyset\}.
\]
(Here and in the remainder of the paper, $U_\sigma \od \bigcap_{i\in \sigma} U_i$). Equivalently, $\nerve(\U) = \Delta(\code(\U, \R^d))$. A useful result regarding nerves is \emph{Borsuk's nerve lemma \cite[Theorem 10.6]{bjorner95}}, which states that $\nerve(\U)$ is homotopy equivalent to $\bigcup_{i=1}^n U_i$ as topological spaces. The last notion we will make use of is that of a link in a simplicial complex. If $\Delta$ is a simplicial complex and $\sigma\in \Delta$, then the \emph{link} of $\sigma$ is the simplicial complex $\link_\Delta(\sigma) \od \{\tau\subseteq [n]\mid \tau\cap \sigma = \emptyset \text{ and } \tau \cup\sigma \in \Delta\}$. With these tools in hand, we are ready to define and investigate our family of codes. 

\begin{definition}\label{def:Cn}
Let $n \ge 2$. Define $\C_n\subseteq 2^{[2n+2]}$ to be the combinatorial code that consists of the following codewords:\begin{itemize}
\item[(i)] The empty set, 
\item[(ii)] all codewords of the form $\sigma\cup \{n+1\}$ for $\sigma$ a nonempty proper subset of $[n]$.
\item[(iii)] $\{n+1+i\}$ for $1\le i \le n+1$, 
\item[(iv)] $([n]\setminus \{i\})\cup \{n+1\}\cup \{n+1+i\} $ for $1\le i \le n$,
\item[(v)] the codeword $[n]\cup \{n+1\} \cup \{2n+2\}$, and
\item[(vi)] the codeword $\{n+2, n+3, \ldots, 2n+2\}$. 
\end{itemize}
\end{definition}

Note that the maximal codewords in $\C_n$ are those of types (iv), (v), and (vi). 

Some of the important features that would arise in any realization of $\C_n$ are the following:\begin{itemize}
\item The sets $U_{n+2},\cdots, U_{2n+2}$ form a sunflower with $n+1$ petals whose center is disjoint from $U_{n+1}$. This follows from codewords of types (iii) and (vi) in Definition \ref{def:Cn}.
\item $U_{n+1}$ is equal to $U_{1}\cup \cdots\cup U_n$. This follows from the fact that $n+1$ is present in all codewords that involve any $i\in[n]$, together with the fact that $\{n+1\}$ is not a codeword in $\C_n$.
\item The collection $\{U_1,\ldots, U_n\}$ has nerve equal to $2^{[n]}$. This follows from the codeword of type (v) in Definition \ref{def:Cn}, which has $[n]$ as a subset.
\item All the petals $U_{n+2}, \cdots, U_{2n+1}$ touch $U_{n+1}$ at a region corresponding to a unique face of codimension 1 in $2^{[n]}$. This is described by codewords of type (iv) in Definition \ref{def:Cn}). 
\item The petal $U_{2n+2}$ covers the region in $U_{n+1}$ corresponding to the facet of $2^{[n]}$ (this is not pictured in the figure below). This is a result of the codeword of type (v) in Definition \ref{def:Cn} together with the fact that this is the only codeword in $\C_n$ containing $[n]$ as a subset.
\end{itemize}

The code $\C_n$ is not convex, as argued in Theorem \ref{thm:nonconvex}, but one can visualize it as arising from the situation pictured below.\[
\includegraphics[scale=0.9]{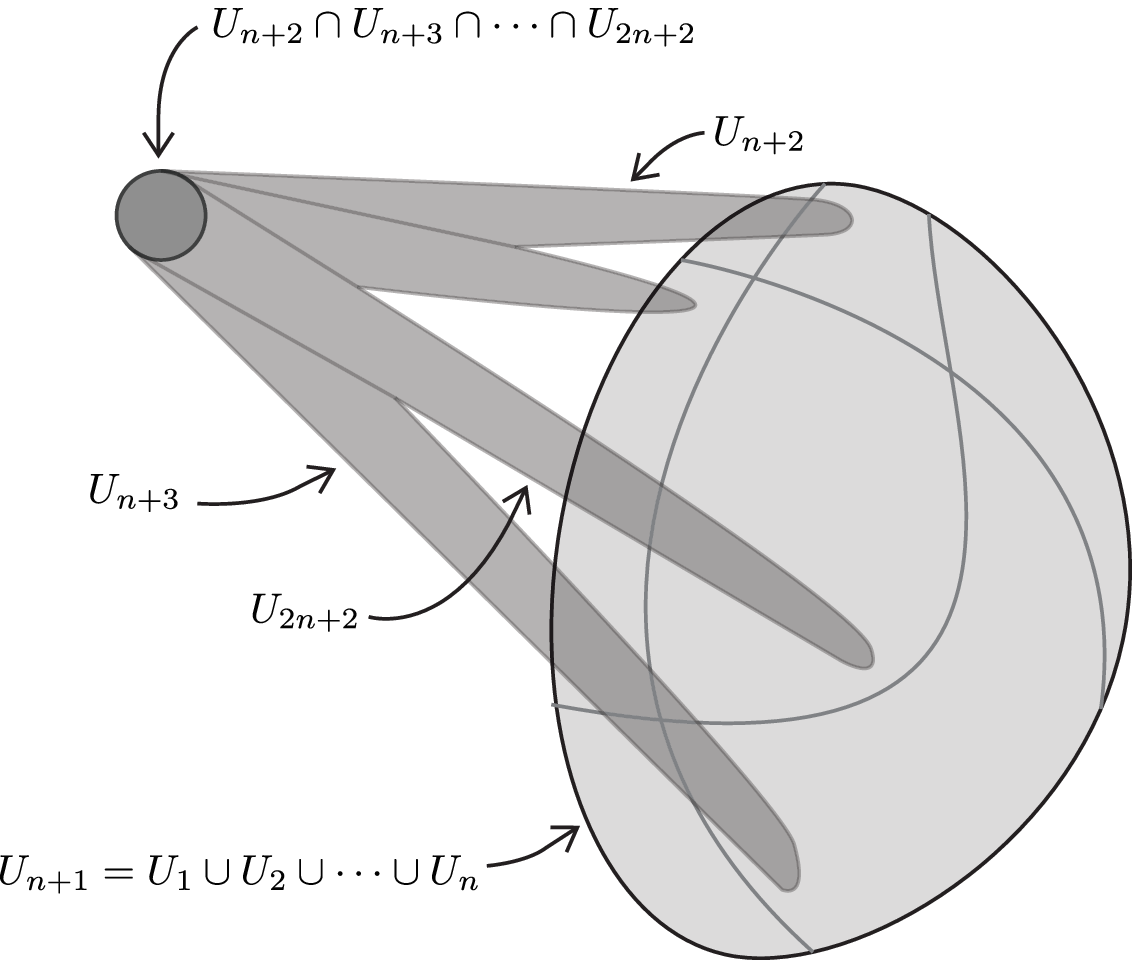}
\]

Essentially, the code $\C_n$ describes a collection of sets which has two pieces: a sunflower consisting of $n+1$ petals, and a collection of sets whose nerve is the $n$-simplex. The petals of the sunflower are incident to the unique facet of the $n$-simplex, as well as all codimension 1 faces,  in the sense that the petals intersect the regions which give rise to these faces in the nerve of $\{U_1,\ldots, U_n\}$. Below we argue that this geometric situation cannot arise from convex open sets.

\begin{theorem}\label{thm:nonconvex}
The code $\C_n$ is not convex.
\end{theorem}
\begin{proof}
Suppose for contradiction that $\C_n$ is convex, with a realization $\{U_1,\ldots, U_{2n+2}\}$ in $\R^d$. For $1\le i \le n$, let $p_i$ be a point in the atom of $([n]\setminus \{i\})\cup \{n+1\}\cup \{n+1+i\} $. Let $C = \conv \{p_i\mid 1\le i \le n\}$. We claim that $C$ contains a point in the atom of $[n]\cup \{n+1\} \cup \{2n+2\}$. To see this, first note that $C\subseteq U_{n+1}$. From the definition of $\C_n$, we see that $\{U_1,\ldots, U_n\}$ cover $U_{n+1}$, and so $\{U_1\cap C, \ldots, U_n\cap C\}$ covers $C$. Now consider the nerve $N$ of $\{U_1\cap C, \ldots, U_n\cap C\}$. By choice of $p_i$, this nerve contains $[n]\setminus \{i\}$ for $1\le i\le n$. This implies that $N$ contains all codimension-one faces of $2^{[n]}$. By the nerve lemma (see \cite[Theorem 10.6]{bjorner95}) $N$ is contractible, which implies that $N$ contains $[n]$, i.e., it is the simplex on $[n]$. Thus there exists a point $p_{n+1}\in C\cap U_1\cap \cdots \cap U_n$. Since the only codeword in $\C_n$ whose support contains $[n]$ is $[n]\cup \{n+1\} \cup \{2n+2\}$, we conclude that $p_{n+1}$ lies in the atom of this codeword.

Now, let $H$ be the affine hull of $\{p_1,\ldots, p_n\}$. Then $H$ contains $p_{n+1}$, and has dimension no larger than $n-1$. Note that $\{U_{n+2},\ldots, U_{2n+2}\}$ is a sunflower of $n+1$ sets, and let $U$ denote its center. If $H$ has positive codimension then  Corollary \ref{cor:plane} implies  that $H$ intersects $U$. If $H$ does not have positive codimension then it is equal to $\R^d$, in which case $H\cap U = U \neq \emptyset$. In either case, the collection $\{U_{n+2}\cap H, \ldots U_{2n+2}\cap H\}$ is a sunflower in $H$. Since the dimension of $H$ is no more than $n-1$, Corollary \ref{cor:conv} implies that $\conv \{p_1,\ldots, p_n\}$ intersects $U$, say at a point $p$. But $\conv \{p_1,\ldots, p_n\}\subseteq U_{n+1}$, and $U = U_{n+2}\cap\cdots \cap U_{2n+2}$. Thus $p$ lies in $U_i$ for $n+1\le i \le 2n+2$. By definition of $\C_n$ this is a contradiction. 
\end{proof}

In \cite{CUR} we introduced the notion of a \emph{nerve obstruction} in a code, generalizing the respective notions from \cite{local15} and \cite{undecidability} of local obstructions and local obstructions of the second kind. Given a code $\C$ and $\sigma\in \widehat{M(C)}\setminus \C$, we say that $\C$ has a nerve obstruction at $\sigma$ if $\link_{\Delta(\C)}(\sigma)$ is \emph{not} the nerve of a collection of convex open sets whose union is convex (that is, if $\link_{\Delta(\C)}(\sigma)$ is not a \emph{convex union representable} complex, as defined in \cite{CUR}). A code is called \emph{locally perfect} if it has no nerve obstructions. For a more comprehensive review of nerve obstructions see \cite{CUR}. 

All convex codes are locally perfect, so nerve obstructions serve as a method of detecting non-convexity. However, we cannot detect the non-convexity of $\C_n$ in this way, as the following theorem argues. 

\begin{theorem}\label{thm:locallyperfect}
The code $\C_n$ is locally perfect. 
\end{theorem}
\begin{proof}
We must check for nerve obstructions at every $\sigma\in \widehat{M(\C_n)}\setminus\C_n$. To do so, we start by determining this set explicitly. The maximal codewords of $\C_n$ are exactly those of types (iv)-(vi) in Definition \ref{def:Cn}, i.e.,\begin{itemize}
\item[(1)] $([n]\setminus \{i\})\cup \{n+1\}\cup \{n+1+i\} $ for $1\le i \le n$,
\item[(2)] $[n]\cup \{n+1\} \cup \{2n+2\},$ and 
\item[(3)] $\{n+2,n+3,\ldots, 2n+2\}$. 
\end{itemize}
We need only examine intersections of these codewords that consist of more than one term. 
Any intersection involving (3) above will be either empty, or equal to $\{n+1+i\}$ for some $1\le i \le n+1$. Definition \ref{def:Cn} states that both possibilities lie in $\C_n$. Intersections of codewords of type (1) and (2) are of the form $\{n+1\}\cup [n]\setminus \sigma$. For $\sigma \neq [n]$ this is a codeword of type (ii) from Definition \ref{def:Cn}, hence it is an element of $\C_n$. For $\sigma = [n]$ this is $\{n+1\}$, which is not an element of $\C_n$. Thus $\{n+1\}$ is the only element of $\widehat{M(\C_n)}\setminus\C_n$. 

The fact that $\C_n$ is locally perfect follows immediately from \cite[Proposition 8.6]{CUR}. Indeed, one can check that the link of $\{n+1\}$ in $\Delta(\C_n)$ is just an $n$-simplex  on the vertex set $[n]\cup \{2n+2\}$, to which we attach $n$ simplices on vertex sets $[n]\setminus\{i\}\cup \{n+1+i\}$ for $i=1,2,\ldots, n$. The proposition cited above implies that such a complex is convex union representable. \end{proof}

We now argue that $\C_n$ is minimally non-convex for all $n$.  We have seen that $\C_n$ is not convex, so we just need to show that every code it covers in $\ParCode$ is convex. Theorem \ref{thm:cover} implies that we just need to check all the simple trunks of $\C_n$ and all $\C_n^{(i)}$ for $1\le i \le 2n+2$. We break these checks up into a number of cases, each treated in one of the lemmas below.

\begin{lemma}\label{lem:Cnsimpletrunks}
All the simple trunks in $\C_n$ are convex.
\end{lemma}
\begin{proof}For $1\le i \le n$, the maximal elements of the simple trunk $\Tk_{\C_n}(i)$ come in two forms: $\{n+1+j\} \cup \{n+1\}\cup [n]\setminus \{j\}$ for $j\neq i$, and $\{2n+2\}\cup \{n+1\} \cup [n]$. One can verify that intersections of these maximal codewords are of the form $\{n+1\} \cup \sigma$ where $\sigma$ is a proper subset of $[n]$ containing $i$. These intersections all lie in $\Tk_{\C_n}(i)$, so the trunk is max-intersection complete and hence convex. 

For $i = n+1$, the trunk consists of codewords of type (ii), (iv), and (v) in Definition \ref{def:Cn}. We can first realize the code consisting of types (ii) and (v) by noting it has a unique maximal codeword, namely $[n]\cup \{n+1\}\cup\{2n+2\}$. This allows us to construct a realization in $\R^2$ in which all atoms are top-dimensional via a construction of \cite{openclosed}. To add codewords of type (iv) we simply define $U_{n+1+i}$ for $1\le i \le n$ to be a convex open subset of the atom for $([n]\setminus \{i\})\cup \{n+1\}$ in this realization. 

For $n+2\le i \le 2n+2$ the trunk $\Tk_{\C_n}(i)$ consists of three codewords, and up to isomorphism it is just the convex code $\{1,2,\emptyset\}$. Thus all simple trunks are convex.
\end{proof}

\begin{lemma}\label{lem:Cn1}
The $i$-th covered code of $\C_n$  is convex for $1\le i \le n+1$
\end{lemma}
\begin{proof}We claim that $\C_n^{(i)}$ is max-intersection complete in these cases. Recall from the proof of Theorem \ref{thm:locallyperfect} that $\{n+1\}$ is the only element of $\widehat{M(\C_n)}\setminus \C$, so we can apply Corollary \ref{cor:Cimic}. 
For $1\le i \le n$, the codeword $\{i,n+1\}$ is present in $\C_n$, and so by the corollary we conclude that $\C_n^{(i)}$ is max-intersection complete in this case. For $i = n+1$, we apply the corollary again, noting that $\emptyset\in \C_n$. Thus $\C_n^{(i)}$ is max-intersection complete and hence convex in these cases. 
\end{proof}

\begin{lemma}\label{lem:Cn2}
The $i$-th covered code of $\C_n$  is convex for $n+2\le i \le 2n+1$
\end{lemma}
\begin{proof}Note that there are only three codewords involving $i$ in $\C_n$: $([n]\setminus \{i-n-1\})\cup \{n+1\} \cup \{i\}$, $\{i\}$, and $\{n+2,n+3,\ldots, 2n+2\}$. The trunks $T_i\cap T_j$ come in three types: those that are equal to $T_i$, which we throw out, those which contain the single codeword $\{n+2,n+3,\ldots, 2n+2\}$, which are already redundant to $T_j \cap T_{2n+2}$, and finally the trunk containing the single codeword $([n]\setminus \{i-n-1\})\cup \{n+1\} \cup \{i\}$. This third possibility is not redundant, and so we see that up to isomorphism, the code $\C_n^{(i)}$ will be the result of removing the index $i$ from the codewords $\{n+2,n+3,\ldots, 2n+2\}$ and $\{i\}$ while leaving the codeword $c = ([n]\setminus \{i-n-1\})\cup \{n+1\} \cup \{i\}$ as well as the rest of the codewords in $\C_n$ unchanged. 

We claim that if we forget the codeword $c$, then $\C_n^{(i)}$ is max-intersection complete. Recall from the proof of Theorem \ref{thm:locallyperfect} that $\{n+1\}$ is the only missing max-intersection in $\C_n$. If  we remove $c$ from $\C_n^{(i)}$, then the set $\{n+1\}$ is no longer an intersection of maximal codewords since the maximal codewords containing it are simply $[n+1]\cup \{2n+2\}$ or those of the form $([n]\setminus \{j\})\cup \{n+1\} \cup\{n+1+j\}$ for $n+1+j\neq i$. The intersection of these maximal codewords is $\{i-n-1, n+1\}$ which is a codeword in $\C_n^{(i)}$. Thus there exists a convex realization of the code $\C_n^{(i)}\setminus \{c\}$.

To account for the codeword $c$ we simply let $U_i$ be a small open ball contained in the region $\big(\bigcap_{j\in c\setminus\{i\}} U_j\big) \setminus U_{i-n-1}$. This region has nonempty interior because it is the nonempty set difference of two convex open sets. Since $c$ is the only codeword of $\C_n^{(i)}$ which contains $i$ this yields a realization of $\C_n^{(i)}$.  Thus $\C_n^{(i)}$ is convex as desired. 
\end{proof}

This leaves only the case that $i=2n+2$. This case is the most complicated, and involves directly constructing a realization of $\C_n^{(i)}$ by defining a number of convex sets using inequalities. A tool which makes this process slightly easier is the following theorem from \cite{openclosed}.

\begin{theorem}[Monotonicity of Convexity, \cite{openclosed}]
Let $\C$ be a convex code. Any code $\D$ with $\C\subseteq \D \subseteq \Delta(\C)$ is also convex. 
\end{theorem}

\begin{lemma}\label{lem:Cn3}
The $i$-th covered code of $\C_n$  is convex for $i=2n+2$.
\end{lemma}
\begin{proof}We first claim that, up to isomorphism, $\C_n^{(2n+2)}$ is the result of deleting index $2n+2$. By Corollary \ref{cor:redundant},  it suffices to show that $T_{2n+2}\cap T_j$ is generated by a set of trunks that does not include $T_{2n+2}$ for all $1\le j\le 2n+1$. For $1\le j \le n+1$, observe that $T_{2n+2}\cap T_j$ consists of the single codeword $[n]\cup \{n+1\}\cup \{2n+2\}$. This is generated as the intersection of $T_1\cap T_2\cap \cdots \cap T_{n+1}$. Similarly, for $n+2\le j \le 2n+1$,  $T_{2n+2}\cap T_j$ consists of the single codeword $\{n+2, n+3,\ldots, 2n+2\}$. This is generated as $T_{n+2}\cap T_{n+3}$. Thus $\C_n^{(2n+2)}$ is the result of deleting $2n+2$ from every codeword where it appears in $\C_n$. 

The resulting code is \emph{not} max-intersection complete, since $\{n+1\}$ is an intersection of maximal codewords, but is not present in the code. Thus we must argue for convexity of $\C_n^{(2n+2)}$ more directly. We leverage monotonicity of convexity by constructing a convex realization of a code $\D\subseteq 2^{[2n+1]}$ with $\D\subseteq \C^{(2n+2)}_n\subseteq \Delta(\D)$. We begin by defining this realization.

First, for $n+2\le j\le 2n+1$ define $U_j\subseteq \R^n$ to be the product of open intervals $(0,1)\times (0,1)\times\cdots \times \R\times \cdots \times (0,1)$ where the factor $\R$ appears in the $(j-n-1)$-th index. That is, $U_j$ is the convex set in $\R^n$ defined by the inequalities $0<x_k<1$ for $k\neq j-n-1$. Observe that this collection of $U_j$ form a sunflower whose center is an open unit hypercube with a vertex at the origin. 

Next, for any interval $(a,b)\subseteq \R$ define a convex open set\[
H_{(a,b)} = \{(x_1,\ldots, x_n)\mid a < x_1+x_2+\cdots + x_n < b\}.
\]
Thus, the set  $H_{(a,b)}$ is just a union of translates of the hyperplane given by the equation $x_1+\cdots +x_n = 0$. Now, let $U_{n+1}$ be the intersection of $H_{(2n-1, 2n)}$ with the positive orthant. Observe that $U_{n+1}$ intersects $U_j$ for $n+2\le j \le 2n+1$, for example at the point $(\frac{1}{2},\frac{1}{2},\ldots, \frac{3n}{2},\ldots, \frac{1}{2})$ where the term $\frac{3n}{2}$ appears in the $(j-n-1)$-th coordinate. However, $U_{n+1}$ does not intersect the center of the sunflower $\{U_j\mid n+2\le j\le 2n+1\}$ since the sum of coordinates of any point in the center is always strictly less than $n$. 

We next define the $U_j$ for $j\in [n]$. For $j\in [n]$, let $H^-_j$ be the open halfspace defined by $\sum_{k\in [n]\setminus \{j\}} x_k > n-1$, and let $U_j = U_{n+1}\cap H_j^-$.

 This realization is illustrated below in the case $n=2$. In the figure $U_3$ is the entire diagonal set, $U_1$ is everything in $U_3$ strictly above $U_4$, and $U_2$ is everything in $U_3$ strictly to the right of $U_5$. 

\[\includegraphics[scale=0.6]{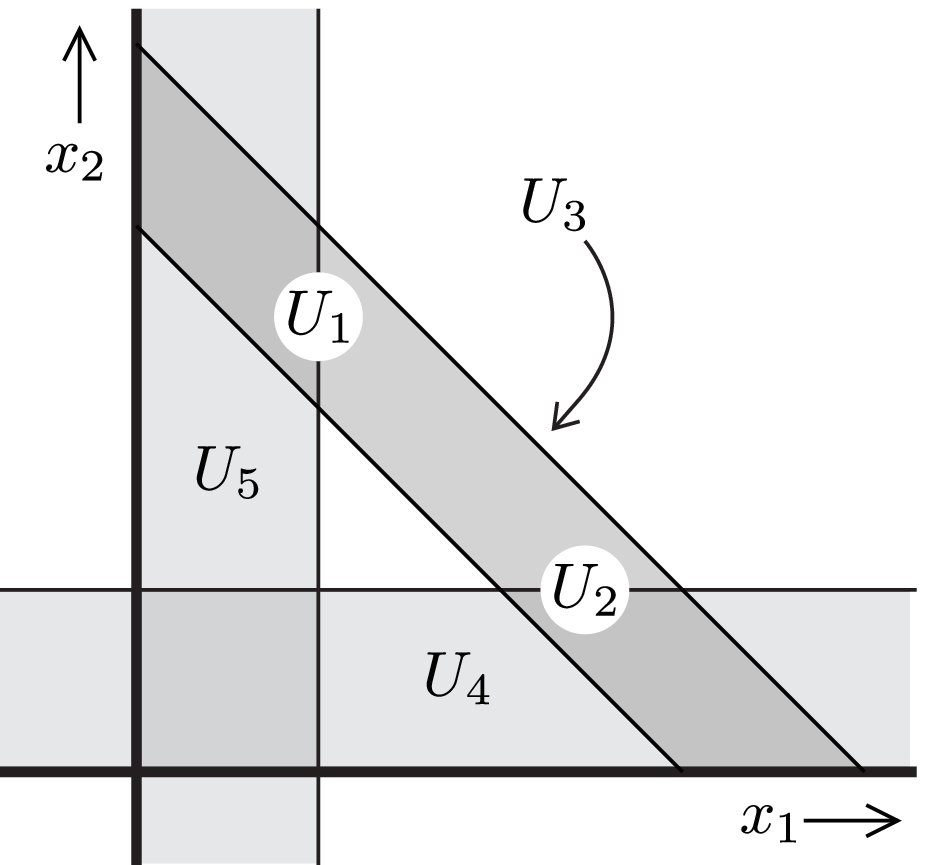}\]

  Let $\D = \code(\{U_1,\ldots, U_{2n+1}\}, \R^n)$. We claim that $\D\subseteq \C_n^{(2n+2)}\subseteq \Delta(\D)$. First, let us describe explicitly the codewords of $\C_n^{(2n+2)}$. These are: \begin{itemize}
\item[(i)] The empty set, 
\item[(ii)] All codewords of the form $\sigma\cup \{n+1\}$ for $\sigma$ a nonempty proper subset of $[n]$.
\item[(iii)] $\{n+1+j\}$ for $1\le j \le n$, 
\item[(iv)] $([n]\setminus \{i\})\cup \{n+1\}\cup \{n+1+j\} $ for $1\le j \le n$,
\item[(v)] the codeword $[n]\cup \{n+1\}$, and
\item[(vi)] the codeword $\{n+2, n+3, \ldots, 2n+1\}$. 
\end{itemize}
Observe that the maximal codewords are those of types (iv), (v), and (vi). 

To see that $\D\subseteq \C_n^{(2n+2)}\subseteq \Delta(\D)$ we determine the codewords of $\D$. We start by focusing on the codewords that  do not contain $n+1$. We claim  that such codewords are precisely those of types (i), (iii), and (vi) in $\C_n^{(2n+2)}$: the empty set arises from any point with all coordinates sufficiently large and negative, codewords of type (iii) arise from points in $U_{n+1+j}$ with large $j$-th coordinate, and the codeword of type (vi) arises from any point in the center of the sunflower $\{U_j\mid n+2\le j\le 2n+1\}$. No other codewords not involving $n+1$ occur in $\D$  since $U_j\subseteq U_{n+1}$ for all $j\in [n+1]$, and the sets $U_j$ for $n+2\le j\le 2n+1$ form a sunflower. Note that all codewords in $\D$ we have examined thus far are codewords of $\C_n^{(2n+2)}$. 

Next we turn to the codewords that involve $U_{n+1}$, or equivalently, we examine the atoms of our realization that are contained in $U_{n+1}$. We make three claims:\begin{itemize}
\item[(1)] For all $j\in [n]$, the set $U_{n+1+j}$ intersects $U_{n+1}$ only at points contained in $U_{[n]\setminus j}\setminus U_j$.
\item[(2)] The sets $U_j$ for $j\in[n]$ cover $U_{n+1}$.
\item[(3)] $U_{[n+1]}$ is nonempty.
\end{itemize}

For (1), we begin by arguing that any point in the intersection $U_{n+1}\cap U_{n+1+j}$ must have $\sum_{k\in [n]\setminus \{j\}} x_k < n-1$, and so it is not in $U_j$. Indeed, this is immediate from the fact that $x_k<1$ for all $k\in [n]\setminus \{j\}$. Moreover, such a point satisfies\[
2n-1< \sum_{k\in[n]} x_k < x_j + n-1
\]
which implies that $x_j> n$. This shows that our point is in $U_k$ for $k\in [n]\setminus \{j\}$, proving the first claim. 

The second claim can be argued as follows: any point in $U_{n+1}$ by definition has \[
2n-1< \sum_{k\in [n]} x_j < 2n
\]
with all $x_j$ positive. This implies that there is $j\in[n]$ with  $x_j<2$. Deleting this $x_j$ from the sum, this yields $\sum_{k\in[n]\setminus \{j\}} x_k >  2n-3$. But since $n\ge2$, we know that $2n-3\ge n-1$. Therefore $\sum_{k\in[n]\setminus \{j\}} x_k > n-1$, and our chosen point lies in $U_j$. This proves that $\{U_j\mid j\in [n]\}$ covers $U_{n+1}$, so the second claim holds. 

Finally, we argue (3). Consider the point whose coordinates are all equal to $2 - \frac{1}{2n}$. This lies in $U_{n+1}$ since the sum of its coordinates is $2n-\frac{1}{2}$. It also lies in $U_j$ for all $j\in [n]$, since the sum of any $n-1$ of its coordinates is equal to $2n+\frac{1}{2n}-\frac{5}{2}$ which is larger than $n-1$ since $n \ge 2$. Thus $U_{[n+1]}$ is nonempty.

Now, (1), (2), and (3) imply that any codeword in $\D$ that contains $n+1$ is one of the codewords of type (ii), type (iv), or type (v) in $\C_n^{(2n+2)}$. From this and our discussion of codewords in $\D$ that do not contain $n+1$, we infer that $\D\subseteq  \C_n^{(2n+2)}$. Moreover, we saw that codewords of types (iv), (v), and (vi) all arise in $\D$. Since these are the maximal codewords of $\C_n^{(2n+2)}$ we conclude that $\C_n^{(2n+2)}\subseteq \Delta(\D)$. Monotonicity of convexity implies that $\C_n^{(2n+2)}$ is a convex code, proving the result.
\end{proof}

\begin{theorem}\label{thm:mnc}
$\C_n$ is a minimally non-convex code.
\end{theorem}
\begin{proof}
We argued in Theorem \ref{thm:nonconvex} that $\C_n$ is not convex. In Lemmas \ref{lem:Cnsimpletrunks}, \ref{lem:Cn1}, \ref{lem:Cn2}, and \ref{lem:Cn3} we argued that every code that $\C_n$ covers in $\ParCode$ is convex. Together these facts imply that $\C_n$ is minimally non-convex. 
\end{proof}

\section{Conclusion}\label{sec:conclusion}

Theorem \ref{thm:sunflower} provides a description of the structure of single sunflower of convex open sets. Informally, this theorem captures a certain ``rigidity" of these sunflowers. This leads to a number of natural questions, some of which are posed below.

\begin{question}
\blue{In Section \ref{sec:proof}, we described how Corollary \ref{cor:conv} can be thought of as a partial converse to the colorful Carath\'eodory theorem. Are there collections of convex open sets which yield partial converses to other classical convexity theorems, such as Tverberg's theorem or generalizations of Helly's theorem?}
\end{question}

\begin{question}
Given several sunflowers of convex open sets, do there exist restrictions on how their petals can intersect? Can such restrictions be expressed independent of the dimension of the sets?
\end{question}

Answers to these questions are interesting in their own right, and in addition they may provide further examples of non-convex combinatorial codes. It is worth noting that the minimally non-convex codes $\C_n$ of Section \ref{sec:Cn} have the following property: they are only missing one max-intersection, which consists of the single index $n+1$. All other known examples of minimally non-convex codes have nerve obstructions (see \cite{CUR}), and these examples have the property that they are again missing a single max-intersection consisting of the empty set. 

This structure raises the following question. Given a minimally non-convex code $\C$, can $\C$ be missing more than one max-intersection? More generally, what is the structure of the max-intersections that are missing from $\C$? We conjecture that minimally non-convex codes  have exactly one missing  max-intersection.

\begin{conjecture}
Let $\C$ be a minimally non-convex code, and let $\widehat{M(\C)}$ be the set of intersections of maximal codewords of $\C$. Then $\widehat{M(\C)}\setminus \C$ contains only one element. 
\end{conjecture}

\section*{Acknowledgements}

\blue{We extend thanks to our referees for insightful feedback and suggestions, including the connection between Corollary \ref{cor:conv} and B\'ar\'any's colorful Carath\'eodory theorem.} We thank Isabella Novik for detailed and illuminating feedback on a number of drafts of this paper. We also thank Zvi Rosen for sharing an independent proof of Theorem \ref{thm:sunflower}. Our work was partially supported by the NSF graduate fellowship grant DMS-1664865. 

\bibliographystyle{plain}
\bibliography{neuralcodereferences}

\begin{thebibliography}{10}

\bibitem{hellytoday}
Nina Amenta, Jes\'us~A. De~Loera, and Pablo Sober\'on.
\newblock Helly's theorem: new variations and applications.
\newblock In {\em Algebraic and geometric methods in discrete mathematics},
  volume 685 of {\em Contemp. Math.}, pages 55--95. Amer. Math. Soc.,
  Providence, RI, 2017.

\bibitem{caratheodorygeneralization}
Imre B\'ar\'any.
\newblock A generalization of carath\'edory's theorem.
\newblock {\em Discrete Mathematics}, 40(2):141--152, 1982.

\bibitem{bjorner95}
Anders Bj\"orner.
\newblock Topological methods.
\newblock In {\em Handbook of combinatorics, {V}ol.\ 1,\ 2}, pages 1819--1872.
  Elsevier Sci. B. V., Amsterdam, 1995.

\bibitem{undecidability}
A.~Chen, F.~Frick, and A.~Shiu.
\newblock Neural codes, decidability, and a new local obstruction to convexity.
\newblock {\em SIAM Journal on Applied Algebra and Geometry}, 3(1):44--66,
  2019.

\bibitem{openclosed}
Joshua Cruz, Chad Giusti, Vladimir Itskov, and Bill Kronholm.
\newblock On open and closed convex codes.
\newblock {\em Discrete \& Computational Geometry}, 61:247--270, 2016.

\bibitem{local15}
Carina Curto, Elizabeth Gross, Jack Jeffries, Katherine Morrison, Mohamed Omar,
  Zvi Rosen, Anne Shiu, and Nora Youngs.
\newblock What makes a neural code convex?
\newblock {\em SIAM J. Appl. Algebra Geom.}, 1(1):222--238, 2017.

\bibitem{neuralring13}
Carina Curto, Vladimir Itskov, Alan Veliz-Cuba, and Nora Youngs.
\newblock The neural ring: an algebraic tool for analyzing the intrinsic
  structure of neural codes.
\newblock {\em Bull. Math. Biol.}, 75(9):1571--1611, 2013.

\bibitem{hellyrelatives}
Ludwig Danzer, Branko Gr\"unbaum, and Victor Klee.
\newblock Helly's theorem and its relatives.
\newblock In {\em Proc. {S}ympos. {P}ure {M}ath., {V}ol. {VII}}, pages
  101--180. Amer. Math. Soc., Providence, R.I., 1963.

\bibitem{eckhoff}
J\"urgen Eckhoff.
\newblock Helly, {R}adon, and {C}arath\'eodory type theorems.
\newblock In {\em Handbook of convex geometry, {V}ol.\ {A}, {B}}, pages
  389--448. North-Holland, Amsterdam, 1993.

\bibitem{morphisms}
{R. Amzi} Jeffs.
\newblock Morphisms of neural codes.
\newblock {\em ArXiv e-prints}, 1806.02014. 2018.

\bibitem{CUR}
{R. Amzi} Jeffs and Isabella Novik.
\newblock {Convex union representability and convex codes}.
\newblock {\em International Mathematics Research Notices}, 2019.
\newblock To appear.

\bibitem{obstructions}
Caitlin Lienkaemper, Anne Shiu, and Zev Woodstock.
\newblock Obstructions to convexity in neural codes.
\newblock {\em Adv. in Appl. Math.}, 85:31--59, 2017.

\bibitem{matousek}
Ji\v{r}\'\i\ Matou\v{s}ek.
\newblock {\em Lectures on discrete geometry}, volume 212 of {\em Graduate
  Texts in Mathematics}.
\newblock Springer-Verlag, New York, 2002.

\end{thebibliography}

\end{document}